\documentclass{amsart}
\usepackage{amsmath}
\usepackage{amssymb}
\usepackage{amsfonts}

\usepackage[dvips]{graphicx,color}

\usepackage{amsthm}
\usepackage{amscd}
\usepackage{amsfonts}
\usepackage[all]{xy}

\DeclareMathOperator{\Diff}{Diff}
\DeclareMathOperator{\Hom}{Hom}
\DeclareMathOperator{\Ima}{Im}
\DeclareMathOperator{\Ker}{Ker}

\DeclareMathOperator{\Aut}{Aut}

\DeclareMathOperator{\GL}{GL}
\DeclareMathOperator{\Sh}{Sh}

\DeclareMathOperator{\ad}{ad}
\DeclareMathOperator{\Der}{Der}
\DeclareMathOperator{\QDer}{QDer}
\DeclareMathOperator{\IAut}{IAut}

\DeclareMathOperator{\QAut}{QIAut}
\DeclareMathOperator{\Prim}{Prim}
\DeclareMathOperator{\Met}{Met}
\DeclareMathOperator{\sgn}{sgn}

\def\id{\mathrm{id}}
\def\S{\mathcal{S}}

\def\I{\mathcal{I}}
\def\M{\mathcal{M}}
\def\m{\mathfrak{m}}
\def\f{\mathfrak{f}}
\def\r{\mathbb{R}}

\def\z{\mathbb{Z}}

\def\pt{\partial}

\def \HD#1#2#3#4{
\xymatrix{{#1}\  \ar@<0.5ex>[r]^{{#2}} & \ {#4}
\ar@<0.5ex>[l]^{{#3}}}}

\theoremstyle{definition}
  \newtheorem{thm}{Theorem}[section]
  
  \newtheorem{lem}[thm]{Lemma}

\theoremstyle{definition}
  \newtheorem{defi}[thm]{Definition}

  \newtheorem{rem}[thm]{Remark}

\title[Homotopy theory of $C_\infty$-algebras and characteristic classes]{Homotopy theory of $C_\infty$-algebras and characteristic classes of fiber bundles}

\author{Hiroshige Kajiura}
\address{Faculty of Sciences, Chiba university, 
1-33 Yayoi-cho, Inage-ku, Chiba 263-8522, Japan}
\email{kajiura@math.s.chiba-u.ac.jp}

\author{Takahiro Matsuyuki}
\address{Department of Mathematics, 
Tokyo Institute of Technology, 
2-12-1 Oh-okayama, Meguro-ku, Tokyo 152-8551, Japan.}
\email{matsuyuki.t.aa@m.titech.ac.jp}

\author{Yuji Terashima}
\address{Department of Mathematics, Tohoku University, 
6-3, Aoba, Aramaki-aza, Aoba-ku, Sendai-shi, Miyagi 980-8578, Japan.}
\email{yujiterashima@tohoku.ac.jp}

\date{}

\begin{document}
\maketitle
\begin{abstract}
In this paper we give a Chern-Weil-type construction of characteristic classes of fiber bundles, based on homotopy theory of C-infinity algebras. Our idea is to replace a family of closed manifolds to a family of C-infinity morphisms with family of metrics.
\end{abstract}
\section{Introduction}

In this paper we give a new Chern-Weil type construction of characteristic classes for fiber bundles by using homotopy theory of $C_\infty$-algebras. For vector bundles, there are several constructions of characteristic classes. In particular, Chern-Weil theory 
is a beautiful theory to get characteristic classes with additional objects as connections or metrics. For fiber bundles which are not principal bundles with compact Lie groups, it is difficult to get Chern-Weil-type theory because 
fiber bundles have diffeomorphism groups as structure groups. Our idea is to replace a family of closed manifolds as fibers to a family of $C_\infty$-morphisms with family of metrics on the fibers. Then, we get get Lie algebra valued differential forms on the base manifold and characteristic classes from Maurer-Cartan forms on moduli spaces of $C_\infty$-morphisms.

In order to carry out this program, we need the basics of
the homotopy theory on $C_\infty$-algebras. The homotopy theory on 
$C_\infty$-algebras is a modification of homotopy theory on 
$A_\infty$-algebras which was developed for instance in 
\cite{CL,fukaya:lagrangian,GuMu,hk:thesis,KaTe,KoSo:note,
Le-Ha,lyuba:Ainftycat,smirnov}. 
Remark that \cite{fukaya:lagrangian} treats filtered $A_\infty$ algebras which 
are important for an application to Floer homology and mirror symmetry. 

A strong tool is the decomposition theorem which states that 
any $C_\infty$-algebra is $C_\infty$-isomorphic to the direct sum of a minimal $C_\infty$-algebra and a linear contractible $C_\infty$-algebra. 
The decomposition theorem was first mentioned in \cite{kont:defquant-pub} 
for $L_\infty$-algebras. 
A proof was given in \cite{KaTe} and was presented in \cite{hk:thesis}. See \cite{CL} for a filtered version. Our construction is based on facts which are obtained from the decomposition theorem. An important fact that any $C_\infty$-quasi-isomorphism has its homotopy inverse. This was first proved in \cite{fukaya:lagrangian} for $A_\infty$-algebras 
by a different method. See \cite{HMS, kadeiD, Pr} for related results about homotopy inverse with another version of homotopy. 
 
Our construction of characteristic classes can be seen as a higher homotopy group version of a construction in 
\cite{MT} which use Chen expansions on fundamental groups.   
For punctured-surface bundles whose fundamental groups are free, 
Chen expansions were used to get Morita-Mumford-Miller classes in \cite{Kaw1, Kaw2} before \cite{MT}. 
Our construction can actually give non-trivial characteristic classes 
also for cases where fibers are not surfaces. 
We give an explicit construction of  such an example where the fiber is $S^2 \times S^3$.

It would be interesting to compare our construction with  
rational homotopy theory on fiber bundles in a wonderful paper \cite{SchSta}. 

This paper is organized as follows. In Section 2, 
we discuss the basics of the homotopy theory on $C_\infty$-algebras
and obtain key facts for our construction of 
characteristic classes from the decomposition theorem.  
In Section 3, we introduce automorphism groups and derivation Lie algebras 
associated to minimal $C_\infty$-algebras. 
Using these tools, in Section 4 
we construct characteristic classes of fiber bundles in the following two cases. 
In Subsection 4.1 we construct them 
for fiber bundles which are homologically trivial. 
This case corresponds to the restriction to the Torelli groups 
when the fiber is a surface. 
In Subsection 4.2, we do the same for a fiber bundle 
whose the fiber is a formal manifold under an additional technical assumption.
In Subsection 4.3, we show that there is a commutative diagram 
to relate these constructions with 
the construction using the fundamental group 
in \cite{MT}. 
In Section 5, we give a non-trivial example of our characteristic classes for a fiber bundle with simply-connected fiber. Specifically, we consider the mapping torus of a special diffeomorphism of $S^2\times S^3$, and calculate our characteristic classes for this bundle.

\noindent
{\bf Acknowledgment.} 
We would like to thank M.~Akaho,  C-H.~Cho, 
K.~Fukaya, T.~Gocho, T.~Kadeishvili, 
A.~Kato,  B.~Keller, T.~Kohno, Y.~Kosmann-Schwarzbach, H.~Ohta and J.~Stasheff for valuable discussions. This work is partially supported by Japan Society for the Promotion of Science (JSPS), Grant-in-Aid for JSPS Research Fellow Grant Number 17J01757, Grants-in-Aid for Scientic Research Grant Number 17K05243, 25400081
and JST CREST Grant Number JPMJCR14D6, Japan.

\section{Homotopy theory of $C_\infty$-algebras}\label{sec:homotopy}

A $C_\infty$-algebra is an $A_\infty$-algebra \cite{jds:hahI,jds:hahII} 
which satisfies a homotopy extension of the graded commutativity condition. 
See for instance \cite{kadei:Cinfty-rh} on a history of $C_\infty$-algebras. 
Various properties of $A_\infty$-algebras then hold true also for $C_\infty$-algebras. 
In this section, 
after briefly recalling the notion of $C_\infty$-algebras and $C_\infty$-morphisms, 
we show that homotopical properties of $A_\infty$-algebras discussed in \cite{KaTe} 
hold in a parallel way for $C_\infty$-algebras. 

\subsection{$\z$-graded vector space}
Throughout Section \ref{sec:homotopy}, all vector spaces are those over a fixed base field $K$ 
of characteristic zero. 
\begin{defi}
Let $V$ be a $\z$-graded vector space. We denote $V^i$ the subspace of elements of $V$ of \textbf{cohomological degree} $i$ and $V_i=V^{-i}$ the subspace of elements of \textbf{homological degree} $i$. Remark that the \textbf{linear dual} $V^*=\Hom(V,\r)$ of $V$ is graded by $(V^*)^i=\Hom(V_i,\r)$. The \textbf{$p$-fold suspension} $V[p]$ of $V$ for an integer $p$ is defined by 
\[V[p]^i:=V^{i+p}.\]
\end{defi}

\begin{defi}\label{Kos}
For distinct elements $v_1,\dots,v_k\in V$ and a permutation $\pi\in\mathfrak{S}_k$, the sign $\epsilon$ defined by the equation on the $\z$-graded $k$-th symmetric power $S^kV$
\[v_1\cdots v_k=\epsilon\cdot v_{\pi(1)}\cdots v_{\pi(k)}\]
is called the \textbf{Koszul sign} of $(v_1,\dots,v_k)\mapsto (v_{\pi(1)},\dots,v_{\pi(k)})$. Similarly the sign $\bar{\epsilon}$ defined by the same equation in the $\z$-graded $k$-th exterior power $\Lambda^kV$ is called the \textbf{anti-Koszul sign}. Note that we have $\bar{\epsilon}=\sgn \pi \cdot \epsilon$.
\end{defi}

\subsection{$C_\infty$-algebras}
We begin by recalling the notions of an $A_\infty$-algebra and an $A_\infty$-morphism. 

\begin{defi}[$C_\infty$-algebra]
Let $A$ be a $\z$-graded vector space and $m=\{m_n:A^{\otimes n}\to A\}_{n\geq1}$ be a family of linear maps with $\deg m_n=2-n$. The pair $(A,m)$ satisfying the following conditions is called a \textbf{$C_\infty$-algebra}. 
\begin{itemize}
\item (\textbf{$A_\infty$-relations}) 
\[\sum_{k+l=n+1}\sum_{j=0}^{k-1}(-1)^{(i+1)(l+1)}m_k\circ (\id_A^{\otimes j}\otimes m_l\otimes \id_A^{\otimes (n-j-l)})=0\]
for $n\geq1$.
\item (\textbf{commutativity}) 
\[\sum_{\sigma\in \Sh(k,n-k)}\bar{\epsilon}\cdot m_n(x_{\sigma(1)},\cdots,x_{\sigma(n)})=0\]
for $n>k>0$ and homogeneous elements $x_1,\dots,x_{n}\in A$. Here $\bar{\epsilon}$ is the anti-Koszul sign for the $(k,n-k)$-shuffle $\sigma$ and elements $x_1,\dots,x_n$.
\end{itemize}
Then $m$ is called a $C_\infty$-structure on $H$.
\end{defi}

The multilinear map $m_k$ has degree $(2-k)$ indicates 
the degree of $m_k(a_1,\dots,a_k)$ is $|a_1|+\cdots +|a_k|+(2-k)$. 
The $A_\infty$-relations imply $(m_1)^2=0$ for $n=1$, 
the Leibniz rule of the differential $m_1$ with respect to the product $m_2$ 
for $n=2$, and the associativity of $m_2$ up to homotopy for $n=3$. 
These facts further imply that 
the cohomology $H(A,m_1)$ has the structure of a 
(non-unital) algebra, where 
the product is induced from $m_2$. 
The commutativity in addition implies $m_2$ is graded commutative, 
and hence $H(A,m_1)$ is graded commutative algebra.

Note that 
the product $m_2$ is strictly associative in $A$ if $m_3$=0. 
\begin{defi}
Let $(A,m)$ be a $C_\infty$-algebra.
\begin{itemize}
\item If higher products are all zero, i.e. $m_3=m_4=\cdots =0$, $(A,m)$ is called a \textbf{differential graded commutative algebra (DGcA)}. 
\item If $m_1=0$, $(A,m)$ is called \textbf{minimal}.  
\end{itemize}
\end{defi}
\begin{rem}[Bar construction of a $C_\infty$-algebra]\label{bar}
Let $(A,m)$ be an $C_\infty$-algebra and $s:A\to A[1]$ be the suspension map. Defining the suspension of $m_n$ by $\bar{m}_n:=s\circ m_n\circ (s^{-1})^{\otimes n}$ for all $n\geq 1$, then the degree of $\bar{m}_n$ is $1$ and the $A_\infty$-relations are rewritten as the simpler equations
\[\sum_{k+l=n+1}\sum_{j=0}^{k-1}\bar{m}_k\circ (\id_{A[1]}^{\otimes j}\otimes \bar{m}_l\otimes \id_{A[1]}^{\otimes (n-j-l)})=0\]
(Getzler-Jones \cite{gj:cyclic}). We denote the bialgebra $BA$ associated to $A$ by
\[BA:=\bigoplus_{n=0}^\infty A[1]^{\otimes n},\]
whose product $\mu:BA\otimes BA\to BA$ and coproduct $\Delta:BA\to BA\otimes BA$ are defined by 
\[\mu(a_1\otimes\cdots\otimes a_k,a_{k+1}\otimes\cdots \otimes a_n):=\sum_{\sigma\in \Sh(k,n-k)}\epsilon\cdot a_{\sigma(1)}\otimes\cdots\otimes a_{\sigma(n)},\]
\[\Delta(a_1\otimes \cdots \otimes a_n):=\sum_{k=0}^n(a_1\otimes\cdots\otimes a_k)\otimes(a_{k+1}\otimes\cdots \otimes a_n)\]
for $a_1,\dots,a_n\in A[1]$. Here $\epsilon$ is the Koszul sign for the $(k,n-k)$-shuffle $\sigma$ and elements $a_1,\dots,a_n$. Then $\bar{m}_n:A[1]^{\otimes n}\to A[1]$ extends to the unique bialgebra derivation $\mathfrak{m}_n:BA\to BA$ by the co-Leibniz rule $\Delta\circ \m_n=(\m_n\otimes \id+\id\otimes\m_n)\circ\Delta$. Setting
\[\m=\sum_{n=1}^\infty\m_n,\]
then $\m$ is a degree $1$ codifferential on $BA$, i.e. $\m^2=0$, from the $A_\infty$-relations of $m$. Furthermore, the commutativity of $m_n$ is equivalent to the Leibniz rule of $\m_n$, i.e., $\m_n\circ \mu=\mu\circ (\m_n\otimes \id+\id\otimes\m_n)$. Thus a $C_\infty$-algebra $(A,m)$ is equivalent to a differential graded bialgebra $(BA,\m)$, which is called the \textbf{bar construction} of $(A,m)$.
\end{rem}

\begin{defi}[$C_\infty$-morphism]
Let $(A,m)$ and $(A',m')$ be $C_\infty$-algebras. A family $f=\{f_n:A^{\otimes n}\to A'\}$ of linear maps with $|f_n|=1-n$ satisfying the following equations is called a \textbf{$C_\infty$-morphism} $f:(A,m)\to (A',m')$.
\begin{itemize}
\item (\textbf{defining equations for an $A_\infty$-morphism})
\[\sum_{\substack{l\geq 1,\\k_1+\cdots+k_l=n}}(-1)^{\sum_{j=1}^lk_j(l-j)+\sum_{\nu<\mu}k_\nu k_\mu}m_l'\circ (f_{k_1}\otimes \cdots\otimes f_{k_l})\]\[=\sum_{\substack{i+1+j=k,\\ i+l+j=n}}(-1)^{1+n+(i+1)(l+1)}f_k\circ (\id_A^{\otimes i}\otimes m_l\otimes \id_A^{\otimes j})\]
for $n\geq1$.
\item (\textbf{commutativity})
\[\sum_{\sigma\in \Sh(k,n-k)}\bar{\epsilon}\cdot f_n(x_{\sigma(1)},\cdots,x_{\sigma(n)})=0\]
for $n>k>0$ and homogeneous elements $x_1,\dots,x_{n}\in A$. Here $\bar{\epsilon}$ is the anti-Koszul sign for the $(k,n-k)$-shuffle $\sigma$ and elements $x_1,\dots,x_n$.
\end{itemize}
If in particular  $f_2=f_3=\cdots=0$, then $f$ is called a \textbf{linear} $C_\infty$-morphism. 
\end{defi}

The defining equation for an $A_\infty$-morphism for $n=1$ implies that 
$f_1:A\to A'$ forms a chain map $f_1:(A,m_1)\to (A',m'_1)$. 
This together with 
the defining equation for $n=2$ implies that 
$f_1:A\to A'$ induces a (non-unital) algebra map 
from $H(A,m_1)$ to $H(A',m_1')$. We denote it by 
$H(f_1):H(A,m_1)\to H(A',m_1')$.

\begin{defi}
A $C_\infty$-morphism $f:(A,m)\to (A',m')$ is 
called a \textbf{$C_\infty$-quasi-isomorphism}  
if $f_1:(A,m_1)\to (A',m_1')$ 
induces an isomorphism between the cohomologies of these two complexes. 
If in particular $f_1$ is itself an isomorphism, then 
$f$ is called a \textbf{$C_\infty$-isomorphism}. 
\end{defi}

\begin{rem}[Bar construction of a $C_\infty$-morphism]
Let $f:(A,m)\to (A',m')$ be a $C_\infty$-morphism. Defining the suspension of $f_n$ by $\bar{f}_n:=s\circ f_n\circ (s^{-1})^{\otimes n}:A[1]^{\otimes n}\to A'[1]$ for all $n\geq 1$, then the degree of $\bar{f}_n$ is $0$ and the relations for $C_\infty$-morphism are rewritten as the equations
\[\sum_{\substack{i\geq 1,\\k_1+\cdots+k_i=n}}\bar{m}_i'\circ (\bar{f}_{k_1}\otimes \cdots\otimes \bar{f}_{k_i})=
\sum_{\substack{i+1+j=k,\\ i+l+j=n}}\bar{f}_k\circ (\id_{A[1]}^{\otimes i}\otimes \bar{m}_l\otimes \id_{A[1]}^{\otimes j}).\]
We can consider the map $\mathfrak{f}:BA\to BA'$ of bialgebras
\[\mathfrak{f}=\sum_{n=1}^\infty\sum_{\substack{i\geq 1,\\k_1+\cdots+k_i=n}}\bar{f}_{k_1}\otimes \cdots\otimes \bar{f}_{k_i}.\]
Note that the commutativity of $f_n$ is equivalent to the compatibility of $\mathfrak{f}$ with the product, i.e., $\mu\circ \mathfrak{f}=\mathfrak{f}\circ \mu$. Then $\mathfrak{f}$ is a map $(BA,\m)\to (BA',\m')$ between bar constructions, i.e. $\mathfrak{f}\circ\m=\m'\circ\mathfrak{f}$ from the first condition of $C_\infty$-morphism. 
\end{rem}
The composition of $C_\infty$-morphisms is defined by the composition of bar constructions of $C_\infty$-morphisms. From the definition, any $C_\infty$-isomorphism has its inverse 
$C_\infty$-isomorphism uniquely. 

On the other hand, 
it is easy to see that 
the composition of $C_\infty$-quasi-isomorphisms 
is a $C_\infty$-quasi-isomorphism. 
A $C_\infty$-quasi-isomorphism has 
its inverse $C_\infty$-quasi-isomorphism 
in a strict sense if and only if it is a $C_\infty$-isomorphism, 
but always has its homotopy inverse as in Theorem \ref{thm:C-infty-homotopy}. 
These facts imply that 
$C_\infty$-quasi-isomorphisms define an equivalence 
relation between $C_\infty$-algebras.

\subsection{Decomposition theorem of $C_\infty$-algebras}
\label{ssec:mmthm}
A pair of minimal $A_\infty$-algebra $(H,m^H)$ and an $A_\infty$-quasi-isomorphism $f:(H,m^H)\to (A,m)$ is called a \textbf{minimal model} of $(A,m)$. 

The decomposition theorem is a strong tool in homotopy algebras. 
It in particular induces the minimal model theorem by Kadeishvili \cite{kadei:minimal}. 
The decomposition theorem is first mentioned in  \cite{kont:defquant-pub} for $L_\infty$-algebras. 
For $A_\infty$-algebras, a proof is given in \cite{KaTe, hk:thesis}. 
See \cite{CL} for a filtered $A_\infty$ version. 
In this subsection, we explain that it holds true also for $C_\infty$-algebras; 
we see that the proof in \cite{KaTe} works without any modification.

As in the case of $A_\infty$-algebras, 
a pair of minimal $C_\infty$-algebra $(H,m^H)$ and an $C_\infty$-quasi-isomorphism $f:(H,m^H)\to (A,m)$ is called a \textbf{minimal model} of $(A,m)$.

\begin{thm}\label{thm:decomp}
Any $C_\infty$-algebra $(A,m)$ is $C_\infty$-isomorphic to 
the direct sum of a minimal $C_\infty$-algebra $M$ and 
a linear contractible $C_\infty$-algebra $C$. Here, a \textbf{linear contractible $C_\infty$-algebra} $C=(C,m^C)$ is a $C_\infty$-algebra such that $m_2^C=m_3^C =\cdots =0$ and the cohomology $H(C,m_1^C)$ is trivial. 
\end{thm}
Especially, 
the composition $M\to M\oplus C\overset{\sim}{\to} A$ gives a minimal model of $(A,m)$, 
where $M\to M\oplus C$ is the linear $C_\infty$-quasi-morphism obtained from the inclusion. 
\begin{proof}
Following \cite{KaTe, hk:thesis}, 
we first choose a Hodge decomposition $(H, \iota,\pi,h)$ of the complex $(A, m_1)$, 
that is, $H:=H(A,m_1)$ is the cohomology, $\iota: H\to A$ and $\pi:A\to H$ are linear map of degree zero 
such that $\pi\circ\iota=id_H$, 
$h:A\to A$ is a linear map of degree minus one and they satisfy 
\begin{equation*}
 m_1 h + h m_1 + P = \id_{A},\qquad h^2=0
\end{equation*}
where $P:=\iota\circ\pi$. This gives a Hodge decomposition of $(BA,\m_1)$, as a complex of vector spaces, 
such that the cohomology is $BH$.  
Actually, $\iota$ and $\pi$ extend to the (linear) coalgebra maps 
$\iota: BH\to BA$ and $\pi: BA\to BH$ 
and one can construct 
a chain homotopy $\mathfrak{h}: BA\to BA$ 
from $\bar{h}, \bar{P}$ and the identity map 
on $A[1]$. 
We in particular choose $\mathfrak{h}$ so that it commutes with the shuffle product $\mu$ 
in the sense that  
$\mathfrak{h}\mu=\mu(\mathfrak{h}\otimes id+ id\otimes \mathfrak{h})$. 
One such chain homotopy $\mathfrak{h}$ shall be constructed 
in equation (\ref{frak-h}).

We put $M:=\Ima P=\Ima \iota$ and $C:=\Ima(m_1h+hm_1)$. Let us consider a coalgebra homomorphism 
$\f^{(2)}:BA\to BA$ defined by $\bar{f}^{(2)}_1=id_{A[1]}$, 
\begin{equation*}
 \bar{f}^{(2)}_2:= \bar{h}\bar{m}_2 - \bar{P}\bar{m}_2\mathfrak{h}, 
\end{equation*}
and $\bar{f}^{(2)}_3=\cdots =0$. 
This defines a $C_\infty$-isomorphism 
$f^{(2)}:(A,m)\to (A,m^{(2)})$, where 
$\m^{(2)}:=\f^{(2)}\circ\m\circ(\f^{(2)})^{-1}$. 
Actually, the commutativity condition for $f^{(2)}$ follows from that 
$\bar{m}_2$ satisfies the commutativity and that $\mathfrak{h}$ commutes with $\mu$. 
In particular, it turns out that 
$m_2^{(2)}=P m_2(P\otimes P)$. 
Thus, $m_2^{(2)}$ defines a bilinear map on $M$. 
Inductively, assume now that 
$(A,m^{(n)})$ is a $C_\infty$-algebra such that 
$m_2^{(n)},\dots,m_n^{(n)}$ define multilinear maps on $M$. 
We set a coalgebra homomorphism 
$\f^{(n+1)}:BA\to BA$ by $\bar{f}_1^{(n+1)}=id_{A[1]}$, 
$\bar{f}_2^{(n+1)}=\bar{f}_3^{(n+1)}=\cdots=\bar{f}_n^{(n+1)}=0$, 
\[
 \bar{f}_{n+1}^{(n+1)}:= \bar{h} \bar{m}^{(n)}_{n+1} - \bar{P}\bar{m}^{(n)}_{n+1}\mathfrak{h}, 
\]
and $\bar{f}_{n+2}^{(n+1)}=\bar{f}_{n+3}^{(n+2)}=\cdots=0$. 
Then, one sees that 
$m_k^{(n+1)}=m_k^{(n)}$ for $k\le n$ and 
$m_{n+1}^{(n+1)}=Pm_{n+1}^{(n)}(P\otimes\cdots\otimes P)$. 
Thus, the induction is completed. 
\end{proof}

The decomposition theorem implies the minimal model theorem as follows. 
Given an $C_\infty$-algebra $(A,m)$ and 
a Hodge decomposition $(H\simeq M,\iota,\pi,h)$ of $(A,m_1)$, 
by the decomposition theorem 
we have an $C_\infty$-algebra structure on $M\oplus C$ and an $C_\infty$-isomorphism $A\simeq M\oplus C$. 
In this situation, the pair $(\iota,\pi)$ extends to 
the pair of linear $C_\infty$-quasi-isomorphisms 
\begin{equation*}
\xymatrix{
 M \ar@<0.5ex>[r]^{\iota\ \ } & M\oplus C \ar@<0.5ex>[l]^{\pi\ \ }   .
}
\end{equation*}
Thus, the composition of $\iota$ with the $C_\infty$-isomorphism 
gives a minimal model  $M\to (A,m)$ of $(A,m)$. What is stronger, 
an $A_\infty$-quasi-isomorphism $(A,m)\to M$ is also obtained here. 

Given a minimal model $(H,m^H)\to (A,m)$, the composition 
$(H,m^H)\to (A,m)\to M$ of the $C_\infty$-quasi-isomorphisms 
is a $C_\infty$-isomorphism since $H\simeq M$. 
Thus, given a $C_\infty$-algebra $(A,m)$ , its minimal models are 
unique (only) up to $C_\infty$-isomorphisms. 
On the other hand, when we choose 
a Hodge decomposition $(H=M,\iota,\pi,h)$ of $(A,m_1)$, 
there exists a canonical construction of a minimal model of $(A,m)$ 
as presented in \cite{KoSo}. 
We employ this fact, too, 
later to construct characteristic classes of fiber bundles.

\subsection{$C_\infty$-homotopy}
In this subsection, we show the existence of 
a homotopy inverse for a $C_\infty$-quasi-isomorphism (Theorem \ref{thm:C-infty-homotopy}). 
We follow \cite{KaTe} where it is shown for the $A_\infty$ case, and rewrite it to the $C_\infty$ case. 
Note that, for the $A_\infty$ case, 
this theorem was first proved in \cite{fukaya:lagrangian} by a different method.  
See \cite{HMS, kadeiD, Pr} for related results about homotopy inverse with another version of homotopy.

For simplicity, suppose $K=\r$ in this subsection.
\begin{defi}
Let $(C,\mu,\Delta)$, $(C',\mu,\Delta')$ be bialgebras, and $f:C\to C'$ be a map of bialgebras. A linear map $D:C\to C'$ satisfying
\[ D\mu=\mu'(f\otimes D+D\otimes f),\quad \Delta' D=(f\otimes D+D\otimes f)\Delta\]
is a \textbf{bialgebra derivation over $f$}. Especially, if $C=C'$ and $f=\id$, $D$ is called a bialgebra derivation on $C$. Then, for a bialgebra derivation $D$ on $C'$, $fD$ is a bialgebra derivation over $f$. If $f$ is an isomorphism of bialgebras, all derivations over $f$ are obtained in such a way. Similarly for a bialgebra derivation $D$ on $C$, $Df$ is a bialgebra derivation over $f$ and the parallel fact holds.
\end{defi}

\begin{defi}[$C_\infty$-homotopy]\label{homotopy}
Two $C_\infty$-morphisms $f,g:(A,m)\to (A',m')$ are \textbf{$C_\infty$-homotopic} if there exist families of $C_\infty$-morphisms $f(t):(A,m)\to (A',m')$ and bialgebra derivations $\mathfrak{h}(t):T(A[1])\to T(A'[1])$ over $\mathfrak{f}(t)$ parametrized piecewise smoothly by $t\in [0,1]$ such that $f(0)=f$, $f(1)=g$ and
\[\frac{d\mathfrak{f}}{dt}(t)=\mathfrak{m}'\circ \mathfrak{h}(t)+\mathfrak{h}(t)\circ \mathfrak{m}.\]
Then we denote $f\sim g$, and $\{(f(t),\mathfrak{h}(t))\}_{t\in [0,1]}$ is called a \textbf{$C_\infty$-homotopy from $f$ to $g$}.
\end{defi}
\begin{thm}\label{thm:C-infty-homotopy}
Let $(A,m)$ and $(A',m')$ be $C_\infty$-algebras. A $C_\infty$-morphism $f:(A,m)\to (A',m')$ is a $C_\infty$-quasi-isomorphism if and only if $f$ is a \textbf{$C_\infty$-homotopy equivalence}, i.e. there exists a $C_\infty$-morphism $g:(A',m')\to (A,m)$ such that $g\circ f\sim \id_A$ and $f\circ g\sim \id_{A'}$.
\end{thm}
\begin{proof}
Given a Hodge decomposition $(H=M,\iota,\pi,h)$ of $(A,m_1)$, by Theorem \ref{thm:decomp} 
we have an $C_\infty$-isomorphism $A\simeq M\oplus C$, and 
the pair $(\iota,\pi)$ extends to the pair of linear $C_\infty$-quasi-isomorphisms 
\begin{equation*}
\xymatrix{
 M \ar@<0.5ex>[r]^{\iota\ \ } & M\oplus C \ar@<0.5ex>[l]^{\pi\ \ }   .
}
\end{equation*}
Here we show that 
the projection $P=\iota\circ\pi$ also extends to the 
linear $C_\infty$-(quasi-iso)morphism $P:M\oplus C\to M\oplus C$ to $M$ and it turns out to be 
$C_\infty$-homotopic to the identity $C_\infty$-(iso)morphism $id_{M\oplus C}$. In fact, setting $P_t:=(1-t)\bar{P}+t\, \id_{A[1]}:A[1]\to A[1]$,  by $\bar{m}_1P_t=P_t\bar{m}_1$ we have
\begin{equation*}
 \begin{split}
\frac{d}{dt}P_t^{\otimes}& =
P_t^{\otimes}\otimes 
(\id_{A[1]}-\bar{P})\otimes P_t^{\otimes} \\ 
 & =
P_t^{\otimes}\otimes
(\bar{m}_1\bar{h}+\bar{h}\bar{m}_1)
\otimes P_t^{\otimes} \\
 & =[\m,P_t^{\otimes}\otimes\bar{h}\otimes P_t^{\otimes}] , 
\end{split}
\end{equation*}
where we express as $P_t^{\otimes}$ the coalgebra map corresponding to $P_t$. By integrating this equation over $[0,1]$, we see that 
$id_{M\oplus C}$ and $P$ is $C_\infty$-homotopic to each other, 
where the map 
\begin{equation}
\mathfrak{h}:=\int_0^1  (P_t^{\otimes}\otimes\bar{h}\otimes P_t^{\otimes})dt:BA\to BA
 \label{frak-h}
\end{equation}
defines a chain homotopy from $\id_{BA}$ to $P^{\otimes}$ 
giving the Hodge decomposition of $(BA,\m_1)$. 
Here, the integration is defined by $\int_0^1 t^k = [t^{k+1}/(k+1)]_0^1=1/(k+1)$. 
One can also confirm that $\mathfrak{h}$ actually forms a bialgebra derivation.

We also choose a Hodge decomposition $(H'=M',\iota',\pi',h')$ of $(A',m_1')$. 
Then we have the following diagram 
of $C_\infty$-algebras and $C_\infty$-quasi-isomorphisms 
\begin{equation*}
\xymatrix{A\ar[r]^{\sim\quad\ }\ar[d]_f&M\oplus C\ar@<0.5ex>[r]^{\pi\ \ \ }& \ar@<0.5ex>[l]^{\iota\ \ \ } M=H(A,m_1)\ar[d]^{f_H}\\
A'\ar[r]^{\sim\quad\ }&M'\oplus C'\ar@<0.5ex>[r]^{\pi'\ \ \ }& 
 \ar@<0.5ex>[l]^{\iota'\ \ \ } M'=H(A',m_1').}
\end{equation*}
and here we define $f_H$ so that the diagram commutes. 
Since any composition of $C_\infty$-quasi-isomorphisms is 
a $C_\infty$-quasi-isomorphism, so is $f_H$.  
Furthermore, since $M$ and $M'$ are minimal $C_\infty$-algebras, 
$f_H$  is actually a $C_\infty$-isomorphism. Thus, there exists the inverse 
$C_\infty$-isomorphism $(f_H)^{-1}$. 
Then we define $g$ by the commutative diagram
\[\xymatrix{A\ar[r]^{\sim\quad\ }&M\oplus C\ar@<0.5ex>[r]& \ar@<0.5ex>[l]
M=H(A,m_1)\\
A'\ar[r]^{\sim\quad\ }\ar@{.>}[u]^g&M'\oplus C'\ar@<0.5ex>[r]& \ar@<0.5ex>[l]M'=H(A',m_1').\ar[u]_{(f_H)^{-1}}}\]
(Note that, in order to construct this $g$ we need the decomposition theorem only, 
not the notion of $C_\infty$-homotopy. )

Now one can show $g\circ f\sim \id_A$ and $f\circ g\sim \id_{A'}$ since 
they correspond to 
$P\sim \id_{M\oplus C}$ on $M\oplus C$ and 
$P'\sim \id_{M'\oplus C'}$ on $M'\oplus C'$, respectively.  
\end{proof}

From Theorem \ref{thm:C-infty-homotopy}, a $C_\infty$-quasi-isomorphism has its homotopy inverse.

\section{Automorphism groups and derivation Lie algebras}
We consider the case of $K=\r$ in this section. For simplicity, the cofree coalgebra $BH$ generated by a $\z$-graded vector space $H$ is denoted by $BH$, which is a bialgebra as in Remark \ref{bar}.
\subsection{Lie groups and their Lie rings}
\label{liealg}
Let $(H,m)$ be a $C_\infty$-algebra. We consider graded Lie subalgebras contained in the graded Lie algebra
\[ \Der^+(BH):=\{D\in\Der(BH); D|_{\r\oplus H[1]}=0\}.\]
The degree $0$ part $\Der^{+}(BH)^0$ of $\Der^+(BH)$ is the Lie ring of the Lie group $\IAut(BH)$ of isomorphisms $\mathfrak{f}:BH\to BH$ of bialgebras such that the lowest term is $f_1=\id_H$. The exponential map $\exp:\Der^{+}(BH)^0\to \IAut(BH)$ is bijective. 

Since $m$ is a minimal $C_\infty$-algebra structure, its bar construction $\mathfrak{m}$ is in $\Der^{+} (BH)^1$. So the inner derivation $\ad(\mathfrak{m})=[\mathfrak{m},-]$ is a degree $1$ differential on $\Der^+(BH)$. 

The group of $C_\infty$-isomorphisms $f:(H,m)\to (H,m)$ such that $f_1=\id_H$ is denoted by $\IAut(H,m)$. Its Lie ring is the Lie algebra of bialgebra derivations which is a chain map $(BH,\mathfrak{m})\to (BH,\mathfrak{m})$, described by 
\[\Der^+(H,m)^0:=\Ker (\ad(\mathfrak{m}):\Der^{+}(BH)^0\to \Der^{+}(BH)^1).\]
The Lie ring of the Lie normal subgroup of $\IAut(H,m)$ 
\[\IAut_0(H,m):=\{f\in \IAut(H,m); \text{$f$ is $C_\infty$-homotopic to the identity}\}\]
is $[\m,\Der (BH)^{-1}]$. 
In fact, for any $f\in \IAut_0(H,m)$, there exist $\mathfrak{f}(t)\in \IAut(H,m)$ and a bialgebra derivation $\mathfrak{h}(t)$ over $\mathfrak{f}(t)$ such that
\[\frac{d\mathfrak{f}}{dt}(t)=[\mathfrak{m},\mathfrak{h}(t)],\] 
$\mathfrak{f}(1)=\mathfrak{f}$ and $\mathfrak{f}(0)=\id$. Then we have
\[\log\mathfrak{f}=\int_0^1\mathfrak{f}(t)^{-1}d\mathfrak{f}(t)=\int_0^1[\mathfrak{m},\mathfrak{f}(t)^{-1}\mathfrak{h}(t)]dt=\left[\m,\int_0^1\mathfrak{f}(t)^{-1}\mathfrak{h}(t)dt\right]\]
and which implies $\mathfrak{f}\in \exp([\m,\Der(BH)^{-1}])$.  
Thus, the Lie ring of the quotient Lie group 
\[\QAut(H,m):=\IAut(H,m)/\IAut_0(H,m)\]
is the Lie algebra $\QDer^+(H,m):=\Ima(\Der^+(H,m)^0\to H^0(\Der(BH),\ad(\mathfrak{m})))$.

\section{Construction of characteristic classes of fiber bundles}
\subsection{Moduli space of $C_\infty$-minimal models}
Let $X$ be an $n$-dimensional oriented closed manifold. It will be a typical fiber of a fiber bundle. We denote the reduced de Rham cohomology of $X$ by \[H:=\tilde{H}_{DR}^\bullet(X)=\bigoplus_{p=1}^\infty H_{DR}^p(X),\]
which is the cohomology of the reduced de Rham complex of $X$
\[A:=\tilde{A}^\bullet(X):=\Ker(A^\bullet(X)\to A^\bullet(*)).\]

For a minimal $C_\infty$-algebra structure $m$ on $H$, \textbf{the moduli space $Q(X,m)$ of $C_\infty$-quasi-isomorphisms over $m$} is the set of $C_\infty$-homotopy classes of $C_\infty$-quasi-isomorphisms $\tau:(H,m)\to A$ such that $\tau_1$ induces the identity map on the their cohomology $H$. 

The Lie group $\QAut(H,m)$ acts on $Q(X,m)$ by 
\[\tau\cdot f:=\tau\circ f\]
for $\tau\in Q(X,m)$, $f\in \QAut(H,m)$. This action is free and transitive since an $A_\infty$-quasi-isomorphism has a homotopy inverse. So $Q(X,m)$ has (the inverse limit of) smooth manifold structure which is isomorphic to $\QAut(H,m)$. 

The set $\mathcal{C}_\infty(X)$ of minimal $C_\infty$-structures $m$ on $H$ such that $Q(X,m)\neq\emptyset$ is parametrized by the space
\[\IAut(H,m)\backslash\IAut(BH).\]
So \textbf{the moduli space of $C_\infty$-minimal models} of the reduced de Rham complex $A$ of $X$
\[
Q(X):=\coprod_{m\in \mathcal{C}_\infty(X)}Q(X,m)
\]
is parametrized by the space
\[Q(X,m)\times_{\IAut(H,m)}\IAut(BH)\]
fixing $m$. It is the space of $C_\infty$-homotopy classes of $C_\infty$-minimal models $\tau:(H,m)\to A$ such that $\tau_1$ induces the identity map on the de Rham cohomology $H$. 

The mapping class group of $X$ 
\[\M(X):=\Diff_+(X)/\Diff_0(X)=\pi_0(\Diff_+(X))\]
acts on $Q(X)$ as follows:
\[[\varphi]\cdot[\tau,m]:=[\varphi\circ\tau\circ|\varphi|^{-1},|\varphi|\circ m\circ|\varphi|^{-1}]\]
for $[\tau,m]\in Q(X)$ and $[\varphi]\in\M(X)$. Here $|\varphi|$ is the induced map  on $H$ from $\varphi$. This action is well-defined since two isotopic diffeomorphisms $\varphi_0,\varphi_1$ of $X$ induce $C_\infty$-homotopic DGcA maps $A\to A$. 

\begin{rem}\label{MettoQ}
According to the Hodge-Kodaira's theorem, a Riemannian metric $\mu$ on $X$ defines a Hodge decomposition of $A$. Thus, using the proof of Theorem \ref{thm:decomp}, we have the canonical map
\[\Met(X):=\{\text{Riemannian metric $T^*X\otimes T^*X\to \r$ on $X$}\}\to Q(X).\]
For a diffeomorphism $\varphi$ of $X$, we can define the pushout
\[\varphi_*\mu:=\mu\circ (\varphi^*\otimes \varphi^*).\]
So the diffeomorphism group $\Diff_+(X)$ acts on $\Met(X)$. Furthermore, the Hodge decomposition obtained from $\varphi_*\mu$ is the image by $(\varphi^*)^{-1}$ of the Hodge decomposition obtained from $\mu$. So the $C_\infty$-minimal model obtained from $\varphi_*\mu$ is equal to $\varphi\cdot M$, where the $C_\infty$-minimal model $M$ is obtained from $\mu$. It means the map $\Met(X)\to Q(X)$ is $\Diff_+(X)$-equivariant (see also \cite{MT}).

\end{rem}

\subsection{Construction}\label{construction}

Let $E\to B$ be a smooth fiber bundle whose fiber is an oriented closed manifold $X$ with base point. It means $E\to B$ is equipped with a section. For simplicity, we set
\[Q:=Q(X),\ \mathcal{C}_\infty:=\mathcal{C}_\infty(X),\ Q(m):=Q(X,m),\ \mathcal{M}:=\mathcal{M}(X).\]
Choose a smooth fiberwise metric $g$ of $E\to B$. The metrics $g_b$ on fiber $E_b$ for $b\in B$ defines a Hodge decomposition on the de Rham complex $A^\bullet(E_b)$. We can restrict this Hodge decomposition to the one of $\tilde{A}^*(E_b)$. It gives a $C_\infty$-minimal model $\tau_b:(\tilde{H}^\bullet_{DR}(E_b),m_b)\to \tilde{A}^\bullet(E_b)$ of fibers as in \cite{KoSo} or Remark \ref{MettoQ}. Then we can obtain the map $B\to \S\backslash Q$ by $b\mapsto [\tau_b,m_b]$, where $\S$ is the image of the structure group of $E\to B$ in $\M$. Here $\S\backslash Q$ plays the role of the usual classifying space of bundles with structure group $\S$. Defining the de Rham complex of $\S\backslash Q$ by $A^\bullet(\S\backslash Q):=A^\bullet(Q)^\S$, we have the map $H_{DR}^\bullet(\S\backslash Q)\to H^\bullet_{DR}(B)$. Since any two metrics can be connected by a segment, this map is independent of the choice 
of a metric.

\subsubsection{Homologically trivial bundles}\label{hom-tri}
We consider the case where the structure group of a fiber bundle acts trivially on the de Rham cohomology group of the fiber. In other words, suppose $\S=\mathcal{I}:=\Ker(\M\to \GL(H))$. Then we have a map $q:B\to \mathcal{C}_\infty$ by giving a smooth fiberwise metric of $E\to B$. Fix $m\in \mathcal{C}_\infty$. Since the topological group $\IAut(H,m)$ is contractible, the pullback $q^*\IAut(BH)\to B$ of the principal $\IAut(H,m)$-bundle $\IAut(BH)\to \mathcal{C}_\infty$ is trivial. Taking a trivialization of the principal bundle, we get the $\I$-equivariant map
\[s:q^*Q=Q(m)\times_{\IAut(H,m)}q^*\IAut(BH)\simeq Q(m)\times \mathcal{C}_\infty\to Q(m).\]
Thus we can obtain the chain map
\[A^\bullet(Q(m))^\I\overset{s^*}{\to}A^\bullet(q^*Q)^\I\to A^\bullet(B).\]
From the action of $\QAut(H,m)$, the space $Q(m)$ has the Maurer-Cartan form $\eta\in A^1(Q(m);\QDer^+(H,m))$. Then we have the chain map
\[\Phi:C_{CE}^\bullet(\QDer^+(H,m))\to A^\bullet(Q(m))^{\I}.\]
Here $C_{CE}^\bullet(\QDer^+(H,m)):=\Lambda^\bullet\Hom(\QDer^+(H,m),\r)$ is the Chevalley-Eilenberg complex of the Lie algebra $\QDer^+(H,m)$ introduced in section \ref{liealg}. The differential $d_{CE}$ of the Chevalley-Eilenberg complex is defined by
\[(d_{CE}c)(D_1,\dots,D_{p+1}):=\sum_{i<j}(-1)^{i+j-1}c([D_i,D_j],D_1,\dots,\hat{D}_i,\dots,\hat{D}_j,\dots,D_{p+1})\]
for $p\geq0$ and $c\in C_{CE}^p(\QDer^+(H,m))$. The chain map $\Phi$ is constructed as follows: for a cochain $c\in C_{CE}^p(\QDer^+(H,m))$, we define
\[\Phi(c):=c(\eta^p)=\sum_{i_1<\dots<i_p}\eta_{i_1}\wedge \cdots\wedge \eta_{i_p}c(b_{i^1}\wedge \cdots\wedge b_{i^p}),\]
where we set
\[\eta=\sum_{i}\eta_ib^i\]
using a (topological) basis $\{b^i\}$ of $\QDer^+(H,m)$. The $p$-form $\Phi(c)$ is $\I$-invariant since $\I$ acts on $H$ trivially. Then $\Phi$ is a chain map by the flatness of $\eta$
\[d\eta+\frac12 [\eta,\eta]=\sum_{i}d\eta_ib_i+\sum_{i<j}\eta_i\wedge \eta_j[b^i, b^j]=0.\]
So we obtain the following:
\begin{thm}\label{htrivial}
Let $E\to B$ be a smooth fiber bundle with oriented closed fiber $X$ whose structure group acts trivially on the real cohomology group of $X$. Then the chain map $\Psi:C^\bullet_{CE}(\QDer^+(H,m))\to A^\bullet(B)$ obtained by the construction above induces the map $\Psi^\#$ between cohomologies which is independent of the choice of a smooth fiberwise metric. 
\end{thm}
For each cohomology class $\alpha$ in $H^\bullet_{CE}(\QDer^+(H,m))$, 
we call the image $c_\alpha(E):=\Psi^\#(\alpha)$ by the $C_\infty$-characteristic 
class of $E$ with label $\alpha$.

\subsubsection{Formal manifold bundles}
We consider the case where $X$ is a formal manifold, i.e. $\mathcal{C}_\infty=\mathcal{C}_\infty(X)$ contains the algebra structure $m$ of $H$, and there exists a decomposition of $S$-modules 
\[\Der^{+}(BH)^0=V\oplus \Der^{+}(H,m)^0,\]
where $S$ is the image of $\S$ in $\GL(H)$ and $V$ is an $S$-submodule of $\Der^{+}(BH)^0$.

By the same discussion of Lemma 3.5 in \cite{MT}, we can obtain the following:
\begin{lem}
The $S$-equivariant principal $\IAut(H,m)$-bundle $\IAut(BH)\to \mathcal{C}_\infty$ is $S$-trivial equivariantly. 
\end{lem}
Then there exists an $\S$-equivariant diffeomorphism
\[Q=Q(m)\times_{\IAut(H,m)}\IAut(BH)\simeq Q(m)\times \mathcal{C}_\infty.\] 
Since the space $\mathcal{C}_\infty$ is also contractible, the space $Q$ is homotopic to $Q(m)$ $\S$-equivariantly. Then, from the Maurer-Cartan form on $Q(m)$, we have the chain map
\[C_{CE}^\bullet(\QDer^+(H,m),S)\to A^\bullet(Q(m))^\S\]
in the same way as subsection \ref{hom-tri}. Here \[C_{CE}^\bullet(\QDer^+(H,m),S):=C_{CE}^\bullet(\QDer^+(H,m))^S\] is the $S$-invariant Chevalley-Eilenberg complex of $\QDer^+(H,m)$.

\begin{thm}\label{formal}
Let $E\to B$ be a smooth fiber bundle with oriented closed formal fiber $X$. Suppose there exists a decomposition of $S$-modules 
\[\Der^{+}(BH)^0=V\oplus \Der^{+}(H,m)^0,\]
where $m$ is the algebra structure of $H$ and $S$ is the image of the structure group in $\GL(H)$. Then the chain map $C_{CE}^\bullet(\QDer^+(H,.m),S)\to A^\bullet(B)$ obtained by the construction above induces the map between cohomologies which is independent of the choice of a smooth fiberwise metric. 

\end{thm}

\subsection{Relation to the construction using the fundamental group}\label{surface}
For any $[\tau,m]\in Q$, we have the dual of the bar construction of $\tau$ 
\[(BA)^*\to (BH)^*=\hat{T}W,\]
where $\hat{T}W$ means the completed tensor product generated by the desuspended reduced real homology group $W:=H^*[-1]$. So composing the chain map $C_\bullet(\Omega X)\to (BA)^*$ obtained by iterated integrals from the cube chain complex of the loop space $\Omega X$, we obtain the chain map
\[C_\bullet(\Omega X;\r)\to (\hat{T}W,\delta),\]
where $\delta:=\mathfrak{m}^*$. The degree $0$ part of (the completion of) map induced to homologies gives
\[\hat{\r}\pi_1=\hat{H}_0(\Omega X;\r)\to H_0(\hat{T}W,\delta)=\hat{T}H_1/I_\delta,\]
where $\pi_1:=\pi_1(X)$, $H_1:=H_1(X;\r)[-1]$ and $I_\delta:=\delta(H_2(X;\r)[-1])$. 
Remark that the map associated with an element in $Q$ which comes from a metric $g$ on $X$ is the Chen expansion determined with $g$ following \cite{GLS}. Then we have the $\M$-equivariant map $\theta:Q\to \Theta(\pi_1)$. Here the definition of the space $\Theta(\pi_1)$ is in \cite{MT}.

Fixing $m$, we have the commutative diagram
\[\xymatrix{T_\tau Q(m)\ar[d]^{\theta_*}\ar[r]&T_1\QAut(H,m)\ar[d]^{\theta_*}\ar@{=}[r]&\QDer^+(H,m)\\
T_{\theta(\tau)}\Theta(\pi_1,I_\delta)\ar[r]&T_1\IAut (\hat{L}H_1/I_\delta)\ar@{=}[r]&\Der^+( \hat{L}H_1/I_\delta).}\]
Here $\hat{L}H_1$ is the completed free Lie algebra generated by $H_1$. So we obtain 
\[\theta_*\eta_1=\theta^*\eta_2,\]
where $\eta_1$ is the Maurer-Cartan form on $Q(m)$ by the action of $\IAut(H,m)$ and $\eta_2$ is the one on $\Theta(\pi_1)$ by the action of $\IAut (\hat{L}H_1/I_\delta)$. 

Thus we obtain the following:
\begin{thm}We have the commutative diagram
\[\xymatrix{H^\bullet_{CE}(\QDer^+(H,m))\ar[r]& H^\bullet_{DR}(B)\\H^\bullet_{CE}(\Der^+( \hat{L}H_1/I_\delta))\ar[u]\ar[ur]&}\]
under the assumption in Theorem \ref{htrivial} and 
\[\xymatrix{H^\bullet_{CE}(\QDer^+(H,m),S)\ar[r]& H^\bullet_{DR}(B)\\H^\bullet_{CE}(\Der^+( \hat{L}H_1/I_\delta),S)\ar[u]\ar[ur]&}\]
under the assumption in Theorem \ref{formal}. 
\end{thm}
It would be interesting to compare our construction with another approach to diffeomorphism groups from noncommutative geometry in \cite{L}.

\section{Example for a simply-connected fiber}

In this section, we shall give a non-trivial example of our characteristic classes and its calculation. 

We consider the case of the product $X=S^2\times S^3$ of the 2-sphere and the 3-sphere. Fix base points $*_2$ and $*_3$ of $S^2$ and $S^3$, respectively. Then $*_X=(*_2,*_3)$ is a base point of $X$. Its real homotopy group $\pi_\bullet(X)\otimes \r$ is 3-dimensional and generated by three elements $\alpha,\beta,[\alpha,\alpha]$ as vector space. Here $\alpha\in \pi_2(X)$ is represented by the embedding $S^2\times *_3\subset X$, $\beta \in \pi_3(X)$ is represented by the embedding $*_2\times S^3\subset X$, and $[\alpha,\alpha]\in \pi_3(X)$ is represented by the composition $S^3\to S^2\subset X$ of the Hopf fibration $S^3\to S^2$ and the embedding $S^2\times *_3\subset X$. Note that $[\alpha,\alpha]\in \pi_3(X)$ is the Whitehead product of two $\alpha$'s.  Thus the real homotopy Lie algebra $\pi^\r_\bullet(X):=\pi_\bullet(X)[-1]\otimes \r$ is generated by two elements $\alpha,\beta$ as Lie algebra.

Let us define the diffeomorphism $\varphi$ of $X$ fixing $*_X$ as follows. Choose a orientation-preserving smooth map $g:(S^3,*_3)\to (SO(3),1)$ representing a generator of $\pi_3(SO(3))\simeq \z$. Then put $\varphi(x,y)=(g(y)x,y)$ for $x\in S^2$, $y\in S^3$. The map induces the identity on the homology of $X$. On the other hand, $\varphi_*:\pi_\bullet(X)\otimes\r\to \pi_\bullet(X)\otimes \r$ is described by
\[\varphi_*\alpha=\alpha,\quad \varphi_*\beta=\beta+[\alpha,\alpha]\]
since the map $S^3\to SO(3)\to S^2$ defined by $y\mapsto g(y)*_2$ is the Hopf fibration.

Then we get the mapping torus \[T(\varphi):=X\times [0,1]/(x,0)\sim (\varphi(x),1),\] and the smooth fiber bundle $X\to T(\varphi)\to S^1=[0,1]/\{0,1\}$. We shall apply our construction to this fiber bundle $T(\varphi)\to S^1$ with fixed section $t\mapsto [*_X,t]$.

First, we shall calculate $\QDer^+(H,m)$, where $m$ is the product of the cohomology ring of $X$. The Quillen model $(LW,\delta)$ of $X$ is described as follows:
\[W=\tilde{H}_\bullet(X;\r)[-1]=\langle x_1,x_2,x_4\rangle,\quad \delta=[x_1,x_2]\frac{\pt}{\pt x_4},\]
where the homological degree of $x_i$ is equal to $i$. (Its homology is non-canonically isomorphic to $\pi^\r_\bullet(X)$.) So the space of derivations on $LW$ with homological degree 0 is spanned by
\[x_1\frac{\pt}{\pt x_1},x_2\frac{\pt}{\pt x_2}, x_4\frac{\pt}{\pt x_4},[x_1,x_1]\frac{\pt}{\pt x_2},[x_2,x_2]\frac{\pt}{\pt x_4},[x_1,[x_1,x_2]]\frac{\pt}{\pt x_4}\]
and the space of derivations on $LW$ with homological degree 1 is spanned by
\[x_2\frac{\pt}{\pt x_1},[x_1,x_1]\frac{\pt}{\pt x_1},[x_1,x_2]\frac{\pt}{\pt x_2},[x_1,x_4]\frac{\pt}{\pt x_4},[x_2,[x_1,x_2]]\frac{\pt}{\pt x_4},[x_1,[x_1,[x_1,x_2]]]\frac{\pt}{\pt x_4}.\]
So, calculating the images of these elements by $\ad(\delta)$, we get
\begin{align*}
\QDer^+(H,m)&\simeq \Ima(\Der^+(LW)_0\cap \Ker(\ad(\delta))\to H_0(\Der(LW),\ad(\delta)))\\&\simeq \left\langle [x_1,x_1]\frac{\pt}{\pt x_2}\right\rangle.\end{align*}
Here note that 
\begin{align*}\Der^+(LW)_0:=&\{D\in \Der(LW)_0;D(W)\subset [LW,LW]\}\\
=&\left\langle[x_1,x_1]\frac{\pt}{\pt x_2},[x_2,x_2]\frac{\pt}{\pt x_4},[x_1,[x_1,x_2]]\frac{\pt}{\pt x_4}\right\rangle,\end{align*}
\[\Der^+(LW)_0\cap \Ker(\ad(\delta))=\Der^+(LW)_0,\ (\Ima \delta)_0= \left\langle [x_2,x_2]\frac{\pt}{\pt x_2},[x_2,[x_1,x_2]]\frac{\pt}{\pt x_4}\right\rangle.\]
In particular, $\QDer^+(H,m)$ is 1-dimensional, and $Q(m)$ is diffeomorphic to $\r$ since $\QAut(H,m)$ acts on $Q(m)$ freely and transitively. By means of the discussion in the beginning of Section \ref{surface}, the $C_\infty$-quasi-isomorphism $f\in Q(m)$ induces the isomorphism between the primitive part of bialgebras
\[\psi^{(r)}:\pi_\bullet^\r(X)=\Prim H^\bullet(BA)^*\simeq \Prim H^\bullet(BH,\m)^*=H_\bullet(LW,\delta).\]
Note that the canonical isomorphism $\pi_\bullet^\r(X)=\Prim H^\bullet(BA)^*$ follows from Chen's theorem \cite{Chen} because $X$ is simply-connected. Since $f_1$ induces the identity of $H$, such an isomorphism must be given by
\[\psi^{(r)}(\alpha)=x_1,\quad \psi^{(r)}(\beta)=x_2+r[x_1,x_1]\]
for some $r\in\r$. Specifically, the constant $r$ is determined by the iterated integral
\[r=\int_{S^3} f_2(X,X)+\frac12 \int_{S^3} f_1(X) f_1(X),\]
where $X\in H^2$ is the dual of $x_1$. It means that $Q(m)$ is parametrized by $\{\psi^{(r)}\}_{r\in\r}$. 

Next, we consider the structure group $G:=\langle \varphi\rangle=\{\varphi^n;n\in\z\}$ of $T(\varphi)\to S^1$. The left action of $G$ on $Q(m)$ is described by
\[\varphi\cdot \psi^{(r)}=\psi^{(r)}\circ \varphi^{-1}=\psi^{(r-1)}.\]
Thus the diffeomorphism $Q(m)\simeq \r$ defined by $\psi^{(r)}\mapsto r$ is equivariant with respect to the action of the group $G\simeq \z$. Furthermore, the Lie group
\[\QAut(H,m)\simeq \{\psi\in \Aut(H_\bullet(LW));\psi(x_1)=x_1,\ \psi(x_2)=x_2+s[x_1,x_1]\ (s\in \r)\},\]
which is isomorphic to the Lie group $\r$, acts on $Q(m)$ by the sum of real numbers. The map defined in Section \ref{hom-tri} gives the isomorphism
\[H_{CE}^\bullet(\QDer(H,m))\simeq H^\bullet_{DR}(G\backslash Q(m)).\]

At last, we shall apply our construction of characteristic classes of $T(\varphi)\to S^1$. Let us denote by $\mu$ the product metric on $X=S^2\times S^3$ of the standard metrics. Using a partition of unity, take a metric $\tilde{\mu}$ on the trivial bundle $X\times [0,1]\to [0,1]$ such that $\tilde{\mu}$ coincides with constant metrics $\mu$ and $\varphi_*\mu$ on $[0,\epsilon]$ and $[1-\epsilon,1]$ for small $\epsilon>0$, respectively. Then $\tilde{\mu}$ induces the metric $\hat{\mu}$ on the fiber bundle $T(\varphi)\to S^1$. Clearly, the metric $\mu$ defines the model $\psi^{(0)}$, while $\varphi_*\mu$ defines the model $\psi^{(0)}\circ \varphi=\psi^{(1)}$. Therefore the map $\bar{\gamma}:S^1\to G\backslash Q(m)\simeq \r/\z$ defined from the metric $\hat{\mu}$ (as in Section \ref{construction}) has the lift $\gamma:[0,1]\to \r$ such that $\gamma(0)=0,\ \gamma(1)=1$. Since the map $\bar{\gamma}$ induces the isomorphism $\bar{\gamma}^*:H^1_{DR}(\r/\z)\simeq H^1_{DR}(S^1)$ on cohomologies, our characteristic map 
\[\Phi:H^1_{CE}(\QDer^+(H,m))\simeq H^1_{DR}(G\backslash Q(m))\simeq H^1_{DR}(S^1)\]
is an isomorphism. The vector space $H^1_{CE}(\QDer^+(H,m))$ is generated by the linear dual of the class of $[x_1,x_1]\pt/\pt x_2$.

In the same way, for a simply-connected oriented manifold $X$ and a diffeomorphism $\varphi$ of $X$ preserving its orientation and fixed point, we can obtain a non-trivial characteristic class of the mapping torus $T(\varphi)\to S^1$ if $\varphi$ acts trivially 
on the rational homology of $X$ but does non-trivially on the rational homotopy group of $X$.

\end{document}